\documentclass[reqno]{amsart}
\usepackage{enumitem}
\usepackage{amssymb,amsmath}
\usepackage{graphicx}
\usepackage{float}

\usepackage{tikz}
\usetikzlibrary{angles,quotes,decorations.markings}\usetikzlibrary{calc,intersections,through,backgrounds}
     \usepackage{tkz-euclide}

\parskip=\smallskipamount

\newtheorem{theorem}{Theorem}[section]
\newtheorem{lemma}[theorem]{Lemma}
\newtheorem{corollary}[theorem]{Corollary}

\theoremstyle{definition}
\newtheorem{definition}[theorem]{Definition}

\newtheorem{remark}[theorem]{Remark}

\newcommand{\C}{\mathbb{C}}

\newcommand{\N}{\mathbb{N}}

\newcommand{\R}{\mathbb{R}}

\newcommand{\ang}{\mathrm{ang}}

\newcommand{\Conv}{{\mathrm{conv}}}
\newcommand{\Span}{{\mathrm{span}}}

\newcommand{\T}{\mathrm{TC}}
\newcommand{\VA}{\mathrm{VarAng}}
\newcommand{\V}{\mathrm{Var}}

\newcommand{\cC}{\mathcal{C}}

\newcommand{\cP}{\mathcal{P}}

\newcommand\Real{\mathrm{Re}}
\newcommand\Imag{\mathrm{Im}}

\def\bar{\overline}

\numberwithin{equation}{section}

\makeatletter
\@namedef{subjclassname@2020}{\textup{2020} Mathematics Subject Classification}
\makeatother

\begin{document}
%\doublespacing
\title[Degree-Bounded Polynomial Convexity of Circular and Smooth Arcs]
{Degree-Bounded Polynomial Convexity of Circular and Smooth Arcs}
\author[Marko Slapar]{Marko Slapar}
\address{Marko Slapar, Faculty of Mathematics and Physics\\ University of Ljubljana \\ Jadranska 19\\1000 Ljubljana, Slovenia\\ \newline
Faculty of Education\\ University of Ljubljana \\ Kardeljeva plo\v{s}\v{c}ad 16\\1000 Ljubljana, Slovenia\\
and\newline
  Institute of Mathematics, Physics and Mechanics\\Jadranska
  19\\1000 Ljubljana, Slovenia}
\email{marko.slapar@fmf.uni-lj.si}
\date{\today}

\thanks{The author was supported by the European Union through the ERC Advanced Grant HPDR, grant agreement No. 101053085, 
awarded to Franc Forstneri\v{c}, and by the research program P1-0291 and research project J1-70033 of ARIS, Republic of Slovenia.}

\begin{abstract} 
Let $d\ge 1$ be an integer. We study $d$-polynomial convexity of smooth Jordan arcs in terms of their total absolute curvature $\T(K)$. We prove that every $\cC^2$-smooth Jordan arc $K\subset\C$ satisfying $\T(K)\le \frac{d-1}{d}\pi$ is $d$-polynomially convex. This bound is sharp: for every $\tau>\frac{d-1}{d}\pi$, there exists a smooth Jordan arc $K\subset\C$ such that $\T(K)<\tau$ and $K$ is not $d$-polynomially convex. We also show that, for $0<\alpha<\pi$, the circular arc $A_\alpha=\{e^{it}:|t|\le \alpha\}$ is $d$-polynomially convex if and only if $\alpha\le \frac{d-1}{d}\pi$.\end{abstract}
\maketitle

\section{Introduction} 

\noindent For a compact set $K\subset \C^n$, its polynomial hull is defined as 
\[\cP(K)=\{z\in\C^n; |P(z)|\le \|P\|_K \text{ for all polynomials }P\}.\] 
The compact set $K$ is called polynomially convex if $\cP(K)=K$. In \cite{S}, we introduced several degree-bounded variants of the polynomial hull by restricting the degrees of the polynomials appearing in its definition. In the present paper, we focus on the most natural of these notions. Let $K$ be a compact set in $\C^n$ and $d\in\N\cup\{\infty\}$. The \textit{d-polynomial hull} of $K$ is the set \[\cP_d(K)=\{z\in\C^n; |P(z)|\le \|P\|_K \text{ for all polynomials }P \text{ with }\deg P\le d\}.\] 
A compact set $K\subset\C^n$ is \textit{d-polynomially convex} if $\cP_d(K)=K$. For $d=\infty$, this definition reduces to the usual notions of polynomial hull and polynomial convexity. A compact set $K\subset\C^n$ is  \textit{ finitely polynomially convex} if $\cP_d(K)=K$ for some $d\in\N$. 

Classical polynomial convexity is studied primarily in higher dimensions, since a compact set $K\subset\C$ is polynomially convex if and only if $\C\setminus K$ is connected. By contrast, $d$-polynomial convexity already exhibits interesting phenomena in one complex dimension, where determining whether a finite set or a Jordan arc is $d$-polynomially convex is already a nontrivial problem \cite{S}. 

We study $d$-polynomial convexity of Jordan arcs in $\mathbb C$
in relation to their total absolute curvature. For a $\cC^2$-smooth Jordan arc $K\subset\C$, let its total absolute curvature be \[\T(K)=\int_K|\kappa|ds,\]
where $\kappa$ is the curvature of $K$. The main result of this paper is the following.
\begin{theorem} \label{T1} Let $d\geq 1$ be an integer, and let $K\subset\C$ be a $\cC^2$-smooth Jordan arc. If \[\T(K)\le\frac{d-1}{d}\pi,\]  
then $K$ is $d$-polynomially convex.   
\end{theorem}
The following corollary is an immediate consequence of the above theorem. 
\begin{corollary} Let $K\subset\C$ be a $\cC^2$-smooth Jordan arc. If $\T(K)<\pi$, then $K$ is finitely polynomially convex.  	
\end{corollary}
It remains unclear what happens at the limiting value $\T(K)=\pi$.
In particular, we do not know whether every $\cC^2$-smooth Jordan arc
with total absolute curvature $\pi$ is finitely polynomially convex.

The curvature bound in Theorem \ref{T1} is sharp in the following sense.
\begin{theorem} \label{T2} 
Let $d\ge 1$ be an integer and let
\[\tau>\frac{d-1}{d}\pi.\]
Then there exists a $\cC^{\infty}$-smooth Jordan arc $K\subset\C$ such that $\T(K)<\tau$ and $K$ is not $d$-polynomially convex.
\end{theorem}

In the second part of the paper, we turn to circular arcs. In \cite[Proposition~16]{S}, we proved that the circular arc $A_\alpha$ is not $d$-polynomially convex when $\alpha>\frac{d-1}{d}\pi,$
but the converse was left open; see \cite[Remark~17]{S}. The following theorem resolves this question and gives the exact threshold for $d$-polynomial convexity of circular arcs.
\begin{theorem} \label{T3} 
Let $d\ge 1$ be an integer, let $0<\alpha<\pi$, and let
\[A_\alpha=\{e^{it};\,-\alpha\le t\le\alpha\}
\] 
be the circular arc. Then $A_\alpha$ is $d$-polynomially convex if and only if
\[\alpha\le\frac{d-1}{d}\pi.\]
\end{theorem}

\section{Preliminaries}
\subsection{Degree-bounded polynomial hulls}

The following lemma gives two characterizations of membership in the $d$-polynomial hull.
\begin{lemma}\label{L}  Let $d\ge 1$ be an integer, let $K\subset\mathbb{C}$ be compact, and let $w\in\C$. Define the Veronese map $\nu_d:\C\to\C^d$ by $\nu_d(z)=(z,z^2,\ldots,z^d)$. 

Then the following are equivalent:
\begin{itemize}
\item[(i)] $w\in\cP_d(K);$
\item[(ii)] there exists a regular Borel probability measure $\mu$ supported on $K$ such that
\[w^k=\int_K z^k\,d\mu(z),\qquad k=0,1,\ldots,d;\]	
\item[(iii)] $\nu_d(w)\in\Conv(\nu_d(K)).$ 
\end{itemize}
\end{lemma}
\begin{proof}
Suppose first that $w\in\cP_d(K)$, and let 
\[E_d=\{P|_K;\ P\in\C[z],\ \deg P\le d\}\subset C(K).\]
Define $L:E_d\to\C$ by $L(P|_K)=P(w)$. Since $w\in\cP_d(K)$, this functional is well defined and satisfies
\[|L(P|_K)|=|P(w)|\le \|P\|_K.\]
Thus $\|L\|\le 1$, and $L(1)=1$, so $\|L\|=1$. By the Hahn-Banach theorem, $L$ extends to a norm-one functional $\tilde L$ on $C(K)$. The Riesz representation theorem then yields a regular Borel measure $\mu$ on $K$ such that
\[\tilde L(f)=\int_K f\,d\mu,\qquad f\in C(K).\]
Since $\|\tilde L\|=\tilde L(1)=1$, the measure $\mu$ is positive and has total mass one. Hence it is a probability measure, and
\[w^k=\tilde L(z^k)=\int_K z^k\,d\mu(z),\qquad k=0,1,\ldots,d.\]
Thus \emph{(i)} implies \emph{(ii)}.
Conversely, suppose that \emph{(ii)} holds. If
\[P(z)=\sum_{k=0}^d a_kz^k,\]
then
\[P(w)=\sum_{k=0}^d a_kw^k=\int_K\sum_{k=0}^d a_kz^k\,d\mu(z)
=\int_K P(z)\,d\mu(z).
\]
Consequently,
\[|P(w)|\le \int_K |P(z)|\,d\mu(z)\le \|P\|_K.\]
Hence $w\in\cP_d(K)$, proving that \emph{(ii)} implies
\emph{(i)}.

It remains to prove the equivalence of \emph{(ii)} and \emph{(iii)}.
If \emph{(ii)} holds, then
\[\nu_d(w)
=\left(\int_Kz\,d\mu(z),\ldots,\int_Kz^d\,d\mu(z)\right)=
\int_K \nu_d(z)\,d\mu(z).
\]
Thus $\nu_d(w)$ is a barycenter of points in $\nu_d(K)$, and therefore
\[\nu_d(w)\in\Conv(\nu_d(K)).\]
Conversely, suppose that
\[\nu_d(w)\in\Conv(\nu_d(K)).\]
Then by Carath\'eodory's theorem, there exist $z_1,z_2,\ldots,z_m\in K$, with $m\le 2d+1$, and weights
\[\lambda_1,\lambda_2,\ldots,\lambda_m>0,\quad \sum_{j=1}^m \lambda_j=1,\] 
such that
\[\nu_d(w)=\sum_{j=1}^m \lambda_j\nu_d(z_j).
\]
Define the probability measure
\[\mu=\sum_{j=1}^m \lambda_j\delta_{z_j}.\]
Comparing the coordinates in the preceding identity yields
\[w^k=\sum_{j=1}^m \lambda_jz_j^k=\int_Kz^k\,d\mu(z),
\qquad k=1,2,\ldots,d.
\]
The identity for $k=0$ follows from
$\sum_{j=1}^m\lambda_j=1$. Hence \emph{(ii)} holds, completing the proof.
\end{proof}

The analogue of the preceding measure characterization for the classical polynomial hull is well known; see, for example, \cite[Section~1.2]{St}. 

\subsection{Variation of angle and total absolute curvature} 

For nonzero vectors $u,v\in\mathbb C$, let
\[\ang(u,v)=\arccos\!\left(
        \frac{\Real(u\overline v)}{|u|\,|v|}\right)\in[0,\pi].\]
Thus $\ang(u,v)$ is the smaller angle between the directions of
$u$ and $v$. Following \cite{M}, we extend the polygonal definition of total absolute curvature to arbitrary parametrized maps; for continuous rectifiable curves, this agrees with the classical definition.
\begin{definition}[Total absolute curvature]\label{D1}
Let $X:[a,b]\to\C$ be an arbitrary map.  For a partition
$P:a=t_0<t_1<\cdots<t_m=b,$
consider the polygonal chain with ordered vertices
$X(t_0),X(t_1),\ldots,X(t_m).$
Delete consecutive repeated vertices.  If $x_0,x_1,\ldots,x_N$
are the remaining vertices, put
\[\T_P(X)=\sum_{j=1}^{N-1}\ang(x_j-x_{j-1},\,x_{j+1}-x_j),\]
with the convention that the sum is $0$ if $N\leq 1$.
The \emph{total absolute curvature} of $X$ is
\[\T(X)=\sup_P \T_P(X)\in[0,\infty],\]
where the supremum is taken over all finite partitions $P$ of $[a,b]$.
\end{definition}

\begin{definition}[Variation of angle]\label{D2}
Let $X:[a,b]\to\mathbb C$ be an arbitrary map.  Its \emph{variation of angle} is
\[\VA(X)=\sup
    \left\{
        \sum_{j=1}^{m}
        \ang\left(X(t_{j-1}),X(t_j)\right);\,
        a\leq t_0<\cdots<t_m\leq b,\;
        X(t_j)\neq 0
    \right\},
\]
where the supremum is taken over all finite ordered sets of
parameters at which $X$ is nonzero.  If $X$ takes fewer than two
nonzero values, we set $\VA(X)=0$.
\end{definition}

\begin{remark}\label{R1}
No continuity or rectifiability is required in the preceding definitions, and either quantity is allowed to be infinite. If $X$ is a continuous rectifiable curve, the above definition of $\T(X)$ agrees with the usual polygonal definition of total absolute curvature.  In particular, if $X$ is a regular $\cC^2$ curve, then
\[\T(X)=\int |\kappa(s)|\,ds.\]
If $X$ is a finite step function, $\T(X)$ is the total curvature of the polygonal chain obtained by joining its successive values by straight segments. Likewise, if $X$ is continuous and nonvanishing and
\[X(t)=|X(t)|e^{i\phi(t)}\]
for a continuous real-valued lift $\phi$ of its argument, then
\[\VA(X)=\V_{[a,b]}\phi.
\]
The definition of $\VA$ remains meaningful when
$X$ vanishes: the zero values themselves are simply omitted, while
nonzero values lying on opposite sides of a zero set may still be
compared in the ordered sum.

For a nonvanishing regular curve $X$, $\VA(X)$ is the length of its spherical projection, while its total absolute curvature $\T(X)$ is the length of the spherical projection of its tangent vectors.
\end{remark}

\begin{lemma}\label{L1} Let $X:[a,b]\to\mathbb C$ be a nonconstant map such that
\[X(a)=X(b)=0.\]
Then
\[\T(X)\ge\pi+\VA(X).\]
\end{lemma}

\begin{proof}
We first prove the corresponding statement for a polygonal chain. Let
\[0=x_0,x_1,\ldots,x_N,x_{N+1}=0,\]
where $x_1,\ldots,x_N\neq0$.  Consecutive repeated vertices may be deleted, so we may assume that $x_i\neq x_{i+1}$. If $N=1$, $\T(0,x_1,0)=\pi$, so the assertion is immediate. Assume now $N\ge 2$. For $i=1,\ldots,N-1$ consider the triangle with vertices $0,x_i,x_{i+1}$.  Denote its angles by $\alpha_i,\beta_i,\gamma_i,$ where $\alpha_i$ is the angle at $0$, $\beta_i$ is the angle at $x_i$, and $\gamma_i$ is the angle at $x_{i+1}$.  Thus
\[\alpha_i=\ang(x_i,x_{i+1}).\]
Let
\[e_0=x_1,\qquad e_i=x_{i+1}-x_i\quad(1\le i\le N-1),\qquad e_N=-x_N
\]
be the successive edge vectors of the polygon. At the first vertex $x_1$, the turning angle is
\[
    \ang(e_0,e_1)=\pi-\beta_1.
\]
At an interior vertex $x_i$, $2\le i\le N-1$, we get
\[\ang(e_{i-1},e_i)\ge \pi-(\gamma_{i-1}+\beta_i).
\]
Finally, at $x_N$,
\[\ang(e_{N-1},e_N)=\pi-\gamma_{N-1}.\]
Adding these inequalities yields
\[\T (0,x_1,\ldots,x_N,0)\ge
    N\pi-\sum_{i=1}^{N-1}(\beta_i+\gamma_i)=\pi+
    \sum_{i=1}^{N-1}
    \ang(x_i,x_{i+1}).
    \]
Degenerate collinear triangles (i.e. $x_i$ and $x_{i+1}$ lie on the same line through the origin) are handled by a small perturbation of one of the vertices and taking the limit.

Now let $X$ be a nonconstant map with $X(a)=X(b)=0$. Choose any finite ordered set $a<t_1<\cdots<t_N<b $ such that $X(t_i)\neq 0$. Set $x_i=X(t_i)$ to get the polygonal chain
\[
    0=X(a),\,X(t_1),\ldots,X(t_N),\,X(b)=0.
\]
Therefore, by the definition of total absolute curvature 
\[
    \T(X) \ge \T \left(0,X(t_1),\ldots,X(t_N),0\right)\\
    \ge \pi+\sum_{i=1}^{N-1} \ang\left(X(t_i),X(t_{i+1})\right).
\]
The points $t_1,\ldots,t_N$ were arbitrary.  Taking the supremum
over all finite ordered sets of parameters at which $X$ is nonzero
gives
\[\T(X)\ge \pi+\VA(X).\]
Since $X$ is nonconstant and $X(a)=X(b)=0$, there is at least one
$t\in(a,b)$ with $X(t)\neq0$, so the argument also covers the case
in which $\VA(X)=0$.
\end{proof}

\begin{lemma}\label{L2}
Let $\gamma:[a,b]\to\mathbb C$ be a regular $\cC^2$ curve, and let $G\in BV([a,b],\mathbb C)$ with $\VA(G)<\infty$. Define
\[
H(t)=\int_a^t G(s)\,d\gamma(s)
     =\int_a^t G(s)\gamma'(s)\,ds.
\]
Then
\[\T(H)\le \VA(G)+\T(\gamma).\]
\end{lemma}

\begin{proof} If $G$ is $\cC^1$ and nonvanishing, the assertion follows immediately from $H'=G\gamma'$ and the subadditivity of variation of angle:
\[\T(H)=\VA(H')\le \VA(G)+\VA(\gamma')=\VA(G)+\T(\gamma).\]
The proof below deals with the additional difficulties caused by zeros
and discontinuities of $G$.

Set
\[
h(t)=G(t)\gamma'(t).
\]
Since $G\in BV$, it is regulated and bounded; since $\gamma'$ is continuous, $h$ is regulated and bounded. Hence $h$ is Riemann integrable, since it is a uniform limit of step functions, and
\[H(t)=\int_a^t h(s)\,ds\]
is Lipschitz, and thus absolutely continuous and rectifiable.

We first show that  $\T(H)\le \VA(h)$. Since $h$ is regulated, for every $n$ there is a partition
\[
a=t_0^{(n)}<\cdots<t_{N_n}^{(n)}=b
\]
such that the oscillation of $h$ on each open subinterval is at most $1/n$; see \cite{F}. Choose tags
$\xi_j^{(n)}\in(t_{j-1}^{(n)},t_j^{(n)})$ and define the tagged step function
\[h_n(t)=h(\xi_j^{(n)})
\quad\text{for }t\in(t_{j-1}^{(n)},t_j^{(n)}],\quad h_n(a)=h(\xi_1^{(n)}).\]
Then
\[\|h_n-h\|_{L^1([a,b])}\longrightarrow 0.\]
Let
\[
H_n(t)=\int_a^t h_n(s)\,ds.
\]
The map $H_n$ is polygonal. Its edge on $(t_{j-1}^{(n)},t_j^{(n)})$ is
\[
\left(t_j^{(n)}-t_{j-1}^{(n)}\right)h(\xi_j^{(n)}).
\]
If $h(\xi_j^{(n)})=0$, this is a zero edge and is deleted when computing total curvature. After all zero edges are deleted, the remaining edge directions are precisely the directions of the nonzero values
\[
h(\xi_{j_1}^{(n)}),\ldots,h(\xi_{j_r}^{(n)})
\]
in their original chronological order. Therefore, by the definition of $\VA(h)$,
\begin{equation}\label{E1}\T(H_n)\le \VA(h)
\qquad\text{for every }n.\end{equation}
Moreover,
\begin{equation}\label{E2}\|H_n-H\|_\infty \le \|h_n-h\|_{L^1}\longrightarrow0.\end{equation}
Fix any inscribed polygon of $H$, determined by
\[a= s_0<\cdots<s_m= b,\] and, as usual, delete consecutive repeated vertices. Thus its chord vectors
\[v_j=H(s_j)-H(s_{j-1})\]
are nonzero. For the same parameter values set
\[v_j^{(n)}=H_n(s_j)-H_n(s_{j-1}).\]
By (\ref{E2}), $v_j^{(n)}\to v_j$. Hence, for all large $n$, the $v_j^{(n)}$ are nonzero and
\[\ang\left(v_j^{(n)},v_{j+1}^{(n)}\right)\longrightarrow \ang(v_j,v_{j+1}).
\]
Therefore the curvature of this fixed inscribed polygon of $H$ is at most
\[\liminf_{n\to\infty}\T(H_n)\le \VA(h)\]
by (\ref{E1}). Taking the supremum over all inscribed polygons of $H$ gives
\[\T(H)\le \VA(h).\]

We now show that $\VA(h)\le \VA(G)+\T(\gamma)$.
Take any ordered finite set
\[a\le t_0<\cdots<t_m\le b \]
with $h(t_j)\ne0$. Since $\gamma$ is regular, $\gamma'(t_j)\ne0$, and therefore also $G(t_j)\ne0$. For nonzero complex numbers the angular distance satisfies
\[
\ang(u_1v_1,u_2v_2)\le \ang(u_1,u_2)+\ang(v_1,v_2).\]
Thus
\begin{align*}
\sum_{j=1}^m \ang\left(h(t_{j-1}),h(t_j)\right)
&\le
\sum_{j=1}^m \ang\left(G(t_{j-1}),G(t_j)\right)\\
&\quad+
\sum_{j=1}^m \ang\left(\gamma'(t_{j-1}),\gamma'(t_j)\right)\\
&\le \VA(G)+\T(\gamma).
\end{align*}
Taking the supremum over all such ordered finite sets gives
\[
\VA(h)\le \VA(G)+\T(\gamma).
\]
This completes the proof.
\end{proof}

\begin{lemma}\label{L3}
Let
\[X:[a,b]\longrightarrow \C\backslash\{0\}
\]
be a regular $\cC^2$ curve. Then
\[\VA(X)<\T(X)+\pi.\]
\end{lemma}

\begin{proof}
Reparametrize $X$ by arclength, so that $|X'(s)|=1$. Since
$X(s)\neq 0$, we may write
\[
    X(s)=r(s)e^{i\phi(s)},
    \qquad
    X'(s)=e^{i\theta(s)},
\]
where $r(s)>0$ and $\phi,\theta$ are continuous and real-valued. Since $X$ is $\cC^2$, we may take $\phi,\theta\in \cC^1$. Moreover,
\[\VA(X)=\int_a^b |\phi'(s)|\,ds,\qquad \T(X)=\int_a^b |\theta'(s)|\,ds.\]
Let
\[q(s)=\ang\left(X(s),X'(s)\right)\in[0,\pi]\]
be the smaller angle between the radial direction and the tangent direction. The function $q$ is Lipschitz and thus absolutely continuous. Let $\psi$ denote locally the signed angle from the radial direction to the tangent direction. Thus, wherever $0<q<\pi$, $\psi=q$ or $\psi=-q.$
Differentiating $X=re^{i\phi}$ and using $|X'|=1$ gives the polar-coordinate identities
\begin{equation}
	r'=\cos\psi,\qquad r\phi'=\sin\psi.
    \label{E11}
\end{equation}
We first prove the non-strict estimate. Suppose that
$0<q(s)<\pi$. If $\sin\psi>0$, then $\psi=q$, and by (\ref{E11}),
$\phi'>0$. Since $\theta=\phi+q$
locally, we have $\theta'=\phi'+q'.$
Hence
\begin{equation}
	|\theta'|\ge \theta'=|\phi'|+q'.
  \label{E12}
\end{equation}
If $\sin\psi<0$, then $\psi=-q$, and therefore $\phi'<0$. Now $\theta=\phi-q$,
so $\theta'=\phi'-q'.$
Consequently,
\begin{equation}|\theta'|\ge -\theta'=|\phi'|+q'.
    \label{E13}
\end{equation}
At points where $q=0$ or $q=\pi$, equation (\ref{E11}) gives
$\phi'=0$. Since an absolutely continuous function has derivative
zero almost everywhere on each of its level sets, we also have
$q'=0$ almost everywhere on $\{q=0\}\cup\{q=\pi\}.$
Thus (\ref{E12}) and (\ref{E13}) imply
\begin{equation}|\phi'|+q'\le |\theta'|\qquad\text{a.e. on }[a,b].
    \label{E14}
\end{equation}
Integrating (\ref{E14}), we obtain

\begin{equation}
    \VA(X)=\int_a^b |\phi'|\,ds\le\int_a^b|\theta'|\,ds+q(a)-q(b)\le \T(X)+\pi .
\label{E15}
\end{equation}

It remains to prove that equality cannot occur. Assume, to the contrary, that

\begin{equation}\VA(X)=\T(X)+\pi.
\label{E16}
\end{equation}
Equality in (\ref{E15}) then forces $q(a)=\pi,\,q(b)=0,$
and equality in (\ref{E14}) for almost every $s$. Let
\[v=\inf\{s\in[a,b];\,q(s)=0\},\]
and let
\[u=\sup\{s\in[a,v);\,q(s)=\pi\}.
\]
By continuity of $q$, we have
\begin{equation*} q(u)=\pi,\qquad q(v)=0,\qquad 0<q(s)<\pi \quad (u<s<v).
   % \label{E18}
\end{equation*}
On $(u,v)$, the sign of $\sin\psi$ is constant. After replacing $X$ by $\bar X$ if necessary, we may assume $\sin\psi>0$ on $(u,v)$. Hence $\psi=q$ and $ \phi'>0$ on $(u,v).$ Put $\varepsilon=\pi-q$.
Then $\varepsilon(u)=0$ and  $0<\varepsilon<\pi$ on $(u,v).$
Since equality holds in (\ref{E12}) almost everywhere, we have $\theta'\ge0 $ almost everywhere. But $\theta'=\phi'+q'=\phi'-\varepsilon',$
so
\begin{equation*}
	\varepsilon'\le \phi'\qquad\text{a.e. on }(u,v).
 %\label{E19}
\end{equation*}
Integrating from $u$ gives
\begin{equation}0<\varepsilon(s)\le\phi(s)-\phi(u).
    \label{E110}
\end{equation}
From (\ref{E11}), since $q=\pi-\varepsilon$,
\[r'=\cos q=-\cos\varepsilon,\qquad r\phi'=\sin q=\sin\varepsilon.\]
Since $\phi'>0$ on $(u,v)$, the function $\phi$ is strictly increasing
there. We may therefore use $\phi$ as a parameter and regard $r$ and
$\varepsilon$ as functions of $\phi$. Therefore
\begin{equation}\frac{d\log r}{d\phi}=\frac{r'}{r\phi'}=-\cot\varepsilon.
    \label{E111}
\end{equation}
For sufficiently small $\varepsilon>0$,
\[\cot\varepsilon\ge \frac{1}{2\varepsilon}.\]
Using (\ref{E110}), for $s$ sufficiently close to $u$ we obtain
\begin{equation}\frac{d\log r}{d\phi}\le-\frac{1}{2\varepsilon}\le-\frac{1}{2\left(\phi-\phi(u)\right)}.
    \label{E112}
\end{equation}

Indeed, integrating (\ref{E112}) from $\phi(u)+\delta$ to a fixed
$\phi_1>\phi(u)$, chosen sufficiently close to $\phi(u)$, gives
\[
\log r(\phi_1)-\log r\bigl(\phi(u)+\delta\bigr)
\le
-\frac12
\int_{\phi(u)+\delta}^{\phi_1}
\frac{d\phi}{\phi-\phi(u)}
=
-\frac12
\log\frac{\phi_1-\phi(u)}{\delta}.
\]
As $\delta\downarrow0$, the right-hand side tends to $-\infty$,
whereas continuity of $r$ and $0<r(u)<\infty$ imply
\[
\log r\left(\phi(u)+\delta\right)\longrightarrow \log r(u),
\]
so the left-hand side converges to the finite number
\[\log r(\phi_1)-\log r(u).\]
This is a contradiction. Therefore
\[\VA(X)<\T(X)+\pi.\]
\end{proof}
The non-strict version of this estimate appears in \cite[Theorem~4]{KY}. In the planar case, their estimate gives $\VA(X)\le \T(X)+q(a)-q(b)\le \T(X)+\pi,$ where $q(t)$ denotes the angle between $X(t)$ and $X'(t)$. The point needed here is that, for a regular $\cC^2$ curve avoiding the origin, the last inequality is in fact strict. Related radial-projection estimates also appear in \cite[Section~8]{EWW}.

\section{A curvature criterion for polynomial convexity}

The proof is based on an iterative construction. The moment conditions
are used to construct successive primitives $F_j$, while
Lemmas~\ref{L1}--\ref{L3} control their angular variation and total
curvature. Related iterative arguments involving variation of angle
and curvature appear in \cite[Section~3.2]{KY}.

\begin{proof}[Proof of Theorem \ref{T1}] If $d=1$, the hypothesis gives $\T(K)=0$, so $K$ is a line segment and $\cP_1(K)=\Conv(K)=K$. We may therefore assume that $d>1$.

Let $\alpha=\T(K),$ and let $\gamma:[a,b]\longrightarrow K$ be an arc-length parametrization of $K$. Suppose that $w\in \cP_d(K)\backslash K.$ Using translation, we may assume $w=0$. In particular, $\gamma(t)\neq 0$  for $a\le  t\leq b$.
By Lemma \ref{L},
\[\nu_d(0)=0\in\operatorname{conv}\nu_d(K).
\]
Applying Carath\'eodory's theorem, we obtain distinct points
\[z_1,z_2,\ldots,z_N\in K, \qquad N\leq 2d+1,
\]
and positive weights
\[
    \lambda_1,\lambda_2,\ldots,\lambda_N>0,
    \qquad
    \sum_{r=1}^N\lambda_r=1,
\]
such that
\begin{equation}
    \sum_{r=1}^N\lambda_r z_r^k=0,\qquad k=1,\ldots,d.
    \label{E31}
\end{equation}
Since $\gamma$ is injective, after reordering the points according to
their occurrence along the arc, we may write
\[z_r=\gamma(t_r),\qquad a\leq t_1<\cdots<t_N\leq b.\]
Set
\[v_r=\lambda_r z_r,\qquad r=1,\ldots,N.\]
Since $0\notin K$, every $v_r$ is nonzero. Moreover, for
$k=0,\ldots,d-1$, equation \eqref{E31} gives
\begin{equation}
    \sum_{r=1}^N v_r z_r^k=\sum_{r=1}^N\lambda_r z_r^{k+1}=0.
    \label{E32}
\end{equation}
In particular,
\[\sum_{r=1}^N v_r=0.\]
Define a finite step map $F_0:\mathbb R\to\mathbb C$ by
\[F_0(t) = \begin{cases}
    0, & t\le a\\ 
    \displaystyle\sum_{t_r<t}v_r, & a<t<b \\
    0, & t\ge b.
\end{cases}
\]
Thus
\[F_0(a)=F_0(b)=0,\]
and the successive jumps of $F_0$ are
\[v_1,\ldots,v_N.\]
The equality $\sum_r v_r=0$ ensures that $F_0$ vanishes again after
the final jump. The distributional derivative of $F_0$ is the finite measure
\[dF_0=\sum_{r=1}^N v_r\,\delta_{t_r}.\]
Consequently, for every $f\in \cC^1([a,b])$,
\[\int_{[a,b]} f\,dF_0=\sum_{r=1}^N f(t_r)v_r=-\int_a^b F_0(t)f'(t)\,dt.\]
Taking $f=\gamma^k$, equation
\eqref{E32} becomes
\begin{equation}
    \int_{[a,b]}\gamma^k\,dF_0=0,
    \qquad
    k=0,\ldots,d-1.
    \label{E33}
\end{equation}
For $j=0,\ldots,d-2$, define recursively
\begin{equation} F_{j+1}(t)=\int_a^t F_j(s)\,d\gamma(s)=\int_a^t F_j(s)\gamma'(s)\,ds.
    \label{E34}
\end{equation}
Thus
\[
    dF_{j+1}=F_j\,d\gamma.
\]
The map $F_0$ is of bounded variation, and each $F_{j+1}$ is
Lipschitz, since $F_j$ is bounded and $\gamma'$ is continuous.
We claim that, for every $j=0,\ldots,d-1$,
\begin{equation}
    F_j(a)=F_j(b)=0
    \label{E35}
\end{equation}
and
\begin{equation}
    \int_{[a,b]}\gamma^k\,dF_j=0,
    \qquad
    k=0,\ldots,d-1-j.
    \label{E36}
\end{equation}
For $j=0$, these statements follow from the definition of $F_0$ and \eqref{E33}. Suppose that they hold for some $j\leq d-2$. For $k=0,\ldots,d-j-2,$ we obtain, using \eqref{E34} and integration by parts,
\begin{align}
    \int_{[a,b]}\gamma^k\,dF_{j+1}
    &=\int_a^b \gamma^k F_j\,d\gamma\nonumber\\
    &=\frac{1}{k+1}
    \int_a^b F_j\,d(\gamma^{k+1})\nonumber\\
    &=-\frac{1}{k+1}\int_{[a,b]}\gamma^{k+1}\,dF_j
    =0.
    \label{E37}
\end{align}
The boundary term vanishes because
\[F_j(a)=F_j(b)=0.\]
Taking $k=0$ in \eqref{E37} gives
\[F_{j+1}(b)-F_{j+1}(a)=\int_{[a,b]}dF_{j+1}=0.\]
Since $F_{j+1}(a)=0$ by definition, we obtain $F_{j+1}(b)=0$. This proves \eqref{E35} and \eqref{E36} by induction.

Every $F_j$ is nonconstant. Indeed, since every $v_r$ is nonzero and
\[\sum_{r=1}^N v_r=0,\]
we have $N\geq2$, and $F_0$ is nonzero on an open interval between
the first two jump points. If $F_j$ is nonzero on an open interval,
then
\[F_{j+1}'(t)=F_j(t)\gamma'(t)\]
almost everywhere on that interval. Since $\gamma$ is regular,
$\gamma'(t)\neq0$, and therefore $F_{j+1}$ is nonconstant. By
induction, every $F_j$ is nonconstant.

By Definition \ref{D1} and Remark \ref{R1}, the total curvature of the finite step map $F_0$ is the total curvature of the polygonal chain whose successive edge vectors are $v_1,\ldots,v_N.$
Therefore
\begin{equation*}\T(F_0)=\sum_{r=1}^{N-1}\ang(v_r,v_{r+1})=\sum_{r=1}^{N-1}\ang(z_r,z_{r+1})\le\VA(\gamma),
\end{equation*}
because $\lambda_r>0$ and $z_r=\gamma(t_r)$. Since $\gamma$ does not meet the origin, Lemma \ref{L3} gives
\begin{equation*}\VA(\gamma)<\T(\gamma)+\pi=\alpha+\pi.    
\end{equation*}
Combining the two inequalities, we obtain
\begin{equation}\T(F_0)<\alpha+\pi.
    \label{E38}
\end{equation}

Since each $F_j$ is nonconstant and satisfies
\[F_j(a)=F_j(b)=0,\]
Lemma \ref{L1} gives
\begin{equation} \T(F_j)\ge \pi+\VA(F_j),\qquad j=0,\ldots,d-1.
    \label{E39}
\end{equation}

The map $F_0$ is a finite step map, so both $\T(F_0)$ and $\VA(F_0)$ are finite. Lemma \ref{L2} applied to \eqref{E34} gives
\[\T(F_1)\le\VA(F_0)+\alpha<\infty.\]
Lemma \ref{L1} then implies that $\VA(F_1)<\infty$. Proceeding inductively, Lemma \ref{L2} may be applied at every stage. For $j=0,\ldots,d-2$, we get
\begin{equation}\T(F_{j+1})\le\VA(F_j)+\alpha.
    \label{E310}
\end{equation}
Combining \eqref{E39} and
\eqref{E310}, we obtain
\begin{equation*} \T(F_j) \ge \T(F_{j+1})+\pi-\alpha,\qquad j=0,\ldots,d-2.
   % \label{eq:curvature-descent}
\end{equation*}
Iterating yields
\[\T(F_0)\ge \T(F_{d-1})+(d-1)(\pi-\alpha).
\]
Finally, (\ref{E39}) gives
\[\T(F_{d-1})\geq\pi.\]
Hence
\begin{equation*}
    \T(F_0)\ge \pi+(d-1)(\pi-\alpha)=d\pi-(d-1)\alpha.
   \end{equation*}
Combining this inequality with \eqref{E38}, we obtain
\[d\pi-(d-1)\alpha < \alpha+\pi.
\]
Therefore
\[\alpha>\frac{d-1}{d}\pi.
\]
This contradicts the hypothesis
\[\T(K)=\alpha\le\frac{d-1}{d}\pi.\]
Thus no point of $\C\backslash K$ belongs to $\cP_d(K)$. Consequently, $K$ is $d$-polynomially convex.
\end{proof}

For $d=2$, the curvature criterion also admits a short elementary proof
that avoids the variation-of-angle machinery. We include it because it
gives a more direct geometric interpretation of this special case.

\begin{proof}[Alternative proof of Theorem \ref{T1} for $d=2$] 
Let $w\in\cP_2(K)$, where $\T(K)\le\frac{\pi}{2}$. Let $\gamma:[0,L]\to\C$ be an arc-length parametrization of $K$ and let $\Theta(s)$ be the continuous lift of the tangent angle. Since 
\[|\Theta(s)-\Theta(t)|\le \T(K)\le\frac{\pi}{2},\] we can, after rotating the curve, assume that
\[0\le \Theta(s)\le \frac{\pi}{2},\quad s\in[0,L].\] 
If \[\gamma(s)=x(s)+iy(s),\]
both 
\[x'(s)=\cos \Theta(s),\quad y'(s)=\sin \Theta(s)\]
are nonnegative and hence $x$ and $y$ are nondecreasing. So 
\begin{equation} \label{EP}(x(s)-x(t))(y(s)-y(t))\ge 0,\quad s,t\in[0,L].\end{equation}
By Lemma \ref{L}, $\nu_2(w)\in\Conv(\nu_2(K)),$
so we have finitely many points $z_j=x_j+iy_j\in K,$ and 
weights
\[\lambda_1,\lambda_2,\ldots,\lambda_N>0,\quad \sum_{j=1}^N\lambda_j=1,\] such that
\[w=\sum_{j=1}^N\lambda_jz_j,\quad w^2=\sum_{j=1}^N\lambda_jz_j^2.\] 
Writing $w=u+iv,$
we have 
\[u=\sum_{j=1}^N\lambda_{j}x_j,\quad v=\sum_{j=1}^N\lambda_{j}y_j,\quad uv=\sum_{j=1}^N\lambda_{j}x_jy_j.\]
By the covariance identity
\[0=\sum_{j=1}^N\lambda_{j}x_jy_j-\left(\sum_{j=1}^N\lambda_{j}x_j\right)\left(\sum_{j=1}^N\lambda_{j}y_j\right)=\frac{1}{2}\sum_{k,j=1}^N\lambda_k\lambda_j(x_k-x_j)(y_k-y_j).\]
Since the weights are positive, (\ref{EP}) implies that
\[(x_k-x_j)(y_k-y_j)=0,\quad j,k=1,\ldots,N.\] 
All points $z_j$ must therefore lie on one horizontal or one vertical line. If all points are equal, this is obvious. So let us assume $z_a\ne z_b$. If $y_a=y_b$ and $x_a\neq x_b$, let $z_j$ be another point. If $y_j\neq y_a$, then we must have both $x_a=x_j$ and $x_b=x_j$, which is not possible. So all points lie on the horizontal line $y=y_a$. The vertical case is analogous.

We now use the remaining identity
\[u^2-v^2=\sum_{j=1}^N\lambda_j(x_j^2-y_j^2)\]
or equivalently
\[\sum_{j=1}^N\lambda_j(x_j-u)^2=\sum_{j=1}^N\lambda_j(y_j-v)^2.\]
If all points lie on the same horizontal line, then  
\[\sum_{j=1}^N\lambda_j(x_j-u)^2=0\] 
and so $x_j=u$ for every $j$. Therefore $z_j=z_k$ for every $j,k$. A similar argument holds if all points lie on a vertical line. 
Since all the points $z_j$ coincide, $w\in K$. 
\end{proof}

\section{Sharpness of the curvature bound}

\begin{proof}[Proof of Theorem \ref{T2}] Let $\tau>\frac{d-1}{d}\pi$ and choose $\alpha$ so that
\[\max\left\{0,1-\frac{\tau}{\pi}\right\}<\alpha<\frac1d.\]
Let $\Phi:\R\to \C$ be defined as 
\[
\Phi(x)=\begin{cases}
    |x|^\alpha ,&x\ge 0 \\
    e^{i\pi\alpha}|x|^\alpha ,& x < 0.
\end{cases}
\]
Let $V_{\alpha}\subset\C$ be the image of $\Phi$ and  	
define the probability measure $\mu$ on $V_{\alpha}$ as the push-forward by $\Phi$ of the Cauchy probability measure 
\[d\rho=\frac{1}{\pi}\frac{1}{1+x^2}\,dx\]
on $\R$. Then, for $k=0,1,\ldots, d$, $\alpha k<1$, and we have 
\begin{align*} \int_{V_\alpha}z^k\,d\mu (z)&=\int_{\R}\Phi(x)^k\,d\rho(x)=\frac{1+e^{i\pi\alpha k}}{\pi}\int_0^\infty \frac{x^{\alpha k}}{1+x^2}\,dx\\
    &=\frac{1+e^{i\pi \alpha k}}{2\cos(\pi\alpha k/2)}=(e^{i\pi\alpha/2})^k.
\end{align*}
So, for $w=e^{i\pi\alpha/2}$, we have
\[\int_{V_\alpha}z^k\,d\mu (z)=w^k,\quad 0\le k\le d.\]
Since $\mu(\{0\})=0$, we may regard $\mu$ as a probability measure on $V_\alpha\backslash\{0\}$. By the Richter--Tchakaloff theorem \cite{R,T}, applied to $(V_{\alpha}\backslash\{0\},\mu)$ and the real functions 
\[\{1,\Real z,\Imag z,\ldots,\Real z^d,\Imag z^d\}\]
there exist finitely many nonzero points $z_1,z_2,\ldots,z_N\in V_{\alpha}$, with $N\le 2d+1$, and positive weights \[\lambda_1,\lambda_2,\ldots,\lambda_N>0,\quad \sum_{j=1}^{N}\lambda_j=1,\]
so that 
\[\sum_{j=1}^N\lambda_jz_j^k=w^k,\quad k=1,2,\ldots,d.\] 
By Lemma \ref{L}, $w\in\cP_d (\{z_1,z_2,\ldots,z_N\})$.
Let $K$ consist of finite portions of the two rays containing all the points $z_j$, joined by a $\cC^{\infty}$-smooth arc inside the wedge which avoids $w$. Orient the resulting arc by moving inward on the ray of argument $\pi\alpha$, then outward on the positive real ray. We can choose the connecting arc so that its tangent angle varies monotonically from the direction $(1+\alpha)\pi$ of the incoming ray to the direction $2\pi$ of the outgoing ray. 

\begin{figure}[H]
\begin{center}\begin{tikzpicture}[
    scale=0.6,
    every node/.style={font=\small},
    point/.style={circle,fill=black,inner sep=2pt},
    contour/.style={line width=1.3pt},
]

\coordinate (O)    at (0,0);
\coordinate (zk)   at (1.35,1.38);
\coordinate (zk1)  at (2.00,0);
\coordinate (z2)   at (3.45,3.50);
\coordinate (z1)   at (4.15,4.22);
\coordinate (top)  at (4.55,4.63);

\coordinate (w)    at (2.30,0.92);

\coordinate (zNm1) at (5.55,0);
\coordinate (zN)   at (7.15,0);

\draw[line width=.7pt] (-1.0,0) -- (8.0,0);
\draw[line width=.7pt] (0,-1.0) -- (0,5.2);

\draw[thin] (O) -- (zk);

\draw[
    contour,
    postaction={decorate},
    decoration={
        markings,
        mark=at position 0.48 with {\arrow{stealth}}
    }
]
(top) -- (zk);

\draw[contour]
(zk)
.. controls (0.95,0.85) and (1.05,0.10) ..
(zk1);

\draw[
    contour,
    postaction={decorate},
    decoration={
        markings,
        mark=at position 0.42 with {\arrow{stealth}}
    }
]
(zk1) -- (7.8,0);

\node[point] at (zk)   {};
\node[point] at (zk1)  {};
\node[point] at (z2)   {};
\node[point] at (z1)   {};
\node[point] at (w)    {};
\node[point] at (zNm1) {};
\node[point] at (zN)   {};

\node[right=4pt, below=1pt] at (z1) {$z_1$};
\node[right=4pt, below=2pt] at (z2) {$z_2$};

\node[right=4pt, below=3pt] at (zk) {$z_k$};
\node[below=7pt] at (zk1) {$z_{k+1}$};

\node[right=5pt] at (w) {$w$};

\node[below=8pt] at (zNm1) {$z_{N-1}$};
\node[below=8pt] at (zN) {$z_N$};
\node[above=7pt] at (4.25,0) {$K$};

\pic[
    draw,
    angle radius=0.7cm,
    angle eccentricity=0.63,
    "$\alpha\pi$"
] {angle = zk1--O--zk};

\end{tikzpicture}
\end{center}
  \caption{Construction of a Jordan arc $K$ with $\T(K)<\tau$ and $w\in\cP_d(K)\backslash K.$}
  \label{F2}
  \end{figure}
  
\noindent  The straight parts contribute no curvature, and hence the total absolute curvature of $K$ is equal to $\pi-\alpha\pi<\tau.$ Since $w\not\in K$ and $w\in\cP_d(K)$, the arc $K$ is not $d$-polynomially convex.
\end{proof}

\section{Circular arcs}
\begin{proof}[Proof of Theorem \ref{T3}] We have shown in \cite[Proposition~16]{S} that $A_\alpha$ is not $d$-polynomially convex if $\alpha>\frac{d-1}{d}\pi$. 

For $d=1$, $A_\alpha$ is not convex for every $\alpha>0$, and hence it is not $1$-polynomially convex. We may therefore assume $d\ge 2$.

Assume $\alpha\le\frac{d-1}{d}\pi$ and suppose that $w\in\cP_d(A_\alpha)\backslash A_\alpha$. 
By Lemma \ref{L}, there exists a probability measure $\mu$ supported on $A_\alpha$ such that 
\[w^k=\int_{A_\alpha} z^k\,d\mu,\qquad k=0,1,\ldots,d.\] 
Since $\cP_d(A_\alpha)\subset\cP_1(A_\alpha)=\Conv(A_\alpha)$, the point $w$ must lie in the interior of the unit disk. Let 
\[d\sigma_w(e^{it})=\frac{1-|w|^2}{|e^{it}-w|^2}\frac{dt}{2\pi}\]
be the Poisson probability measure. Then
\[w^k=\int_{S^1}z^k\,d\sigma_w(z)=\int_{A_\alpha} z^k\,d\mu(z),\qquad k=0,1,\ldots,d\] 
and since both measures are positive and $z^{-1}=\bar z$, we actually have
\begin{equation}\label{E0}
	\int_{S^1}z^k\,d\sigma_w(z)=\int_{A_\alpha} z^k\,d\mu(z),\qquad |k|\le d.
	\end{equation}
Now let $q(z)$ be an arbitrary polynomial of degree at most $d-1$. Since
$\cos t=\frac{e^{it}+e^{-it}}{2}$, the trigonometric polynomial
\[(\cos t-\cos\alpha)|q(e^{it})|^2\] has degree at most $d$. From (\ref{E0}) it follows that
 \begin{equation}\label{E51}
 	\int_{S^1}(\cos t-\cos\alpha)|q(e^{it})|^2\,d\sigma_w(e^{it})=\int_{A_\alpha}(\cos t-\cos\alpha)|q(e^{it})|^2\,d\mu(e^{it})\ge 0.\end{equation} 
 The inequality holds since $\cos t-\cos \alpha\ge 0$ on $A_\alpha$. 
 
 Now let 
 \[H=\Span\{1,z,z^2,\ldots,z^{d-1}\}\subset L^2(\sigma_w).\] We define the operator 
 \[C:H\to H,\quad C=P_HM_{\cos t}|_H,\]
 where $P_H$ is the orthogonal projection $P_H:L^2(\sigma_w)\to H$, and $M_{\cos t}$ is the multiplication by the function $\cos t$. The operator $C$ is Hermitian. Moreover,
 \begin{align*}
 	\langle Cq,q\rangle_{L^2(\sigma_w)}&=\langle P_HM_{\cos t}q,q\rangle_{L^2(\sigma_w)}=\langle M_{\cos t}q,P_Hq\rangle_{L^2(\sigma_w)}=\langle M_{\cos t} q,q\rangle_{L^2(\sigma_w)}\\&=\int_{S^1}(\cos t) |q|^2\,d\sigma_w.
 	\end{align*}
 Hence (\ref{E51}) implies that
 \begin{equation}\label{E55}C\ge(\cos\alpha) I.\end{equation} 
 
 We now represent $C$ in a suitable orthonormal basis and show that this matrix has an  eigenvalue smaller than $\cos\alpha$, yielding the desired contradiction.
  
  The elements 
  \[e_0=1,e_j=\frac{z^{j-1}(z-w)}{\sqrt{1-|w|^2}},\quad j=1,\ldots,d-1\]	
form an orthonormal basis of $H$. Indeed, for $j,k\ge 1$ we have 
\begin{align*}\langle e_j,e_k\rangle_{L^2(\sigma_w)}&=\int_{S^1}\frac{z^{j-1}(z-w)}{\sqrt{1-|w|^2}}\overline{\frac{z^{k-1}(z-w)}{\sqrt{1-|w|^2}}}\frac{1-|w|^2}{|z-w|^2}\,dm(z)\\&=\int_{S^1}z^{j-k}\,dm(z)=\delta_{jk},
\end{align*}
where $dm$ is the normalized Lebesgue measure. Also, for $j\ge 1$
\[\langle e_j,e_0\rangle_{L^2(\sigma_w)}=\int_{S^1}\frac{z^{j-1}(z-w)}{\sqrt{1-|w|^2}}\,d\sigma_w(z)=0\]
and
\[\langle e_0,e_0\rangle_{L^2(\sigma_w)}=1.\]
 To write $C$ in this basis, we calculate $\langle Ce_j,e_k\rangle_{L^2(\sigma_w)}$. If $j,k\ge 1$, we have 
 \begin{align*}
 	\langle Ce_j,e_k\rangle_{L^2(\sigma_w)}&=\int_{S^1}(\cos t)e_j\overline{e_k}\,d\sigma_w(z)\\&=\int_{S^1}\frac{z+z^{-1}}{2}z^{j-k}\,dm(z)\\
 	&=\frac12\int_{S^1}(z^{j-k+1}+z^{j-k-1})\,dm(z)\\
 	&=\begin{cases}
    \frac12,&|j-k|=1\\
    0,& \text{otherwise.}
\end{cases}
\end{align*}
If $j=k=0$ we get
\[\langle Ce_0,e_0\rangle_{L^2(\sigma_w)}=\int_{S^1}\cos t\,d\sigma_w(z)=\int_{S^1}\frac{z+z^{-1}}{2}\,d\sigma_w(z)
=\Real\ w.\]
For $j=1$ and $k=0$
 \begin{align*}
 	\langle Ce_1,e_0\rangle_{L^2(\sigma_w)}&=\int_{S^1}(\cos t)e_1\,d\sigma_w(z)\\&=\frac{1}{2\sqrt{1-|w|^2}}\int_{S^1}(z+z^{-1})(z-w)\,d\sigma_w(z)\\
 	&=\frac{1}{2\sqrt{1-|w|^2}}\int_{S^1}(z^2-zw+1-\bar zw)\,d\sigma_w(z)\\
 	&=\frac{1}{2\sqrt{1-|w|^2}}(w^2-w^2+1-|w|^2)=\frac{\sqrt{1-|w|^2}}{2}.
 \end{align*}
Finally, for $j\ge 2$, $k=0$
\begin{align*}
 	\langle Ce_j,e_0\rangle_{L^2(\sigma_w)}&=\int_{S^1}(\cos t)e_j\,d\sigma_w(z)\\&=\frac{1}{2\sqrt{1-|w|^2}}\int_{S^1}(z+z^{-1})z^{j-1}(z-w)\,d\sigma_w(z)\\
 	&=\frac{1}{2\sqrt{1-|w|^2}}\int_{S^1}(z^{j+1}-z^jw+z^{j-1}-z^{j-2}w)\,d\sigma_w(z)\\
 	&=\frac{1}{2\sqrt{1-|w|^2}}(w^{j+1}-w^{j+1}+w^{j-1}-w^{j-1})=0.
 \end{align*}
 The $d\times d$ matrix for the operator $C$ is thus equal to
 \[C=\begin{bmatrix} \Real\,w &\frac{\sqrt{1-|w|^2}}{2}&0&0&\cdots\\
 \frac{\sqrt{1-|w|^2}}{2}&0&\frac12&0&\cdots\\
 0&\frac12&0&\frac12&\ddots\\
 0&0&\frac12&0&\ddots\\
\vdots&\vdots&\ddots&\ddots&\ddots
\end{bmatrix}.	
\]
Let $T$ be the principal $(d-1)\times (d-1)$ submatrix obtained by deleting the first row and first column
\[T=\begin{bmatrix}
 0&\frac12&0&\cdots\\
 \frac12&0&\frac12&\ddots\\
 0&\frac12&0&\ddots\\
\vdots&\ddots&\ddots&\ddots
\end{bmatrix}.	
\]
The eigenvalues of this Toeplitz matrix are 
\[\lambda_r(T)=\cos\frac{r\pi}{d},\quad r=1,\ldots,d-1\]
with the minimal eigenvalue being
\[\lambda_{\min}(T)=\cos\frac{(d-1)\pi}{d}=-\cos{\frac{\pi}{d}}.\]
An eigenvector corresponding to $\lambda_{\min}(T)$ is
\[\left(\sin\frac{(d-1)\pi}{d},\sin\frac{2(d-1)\pi}{d},\ldots,\sin\frac{(d-1)^2\pi}{d}\right).\]
Let $v$ be the normalization of this eigenvector. Then its first component $v_1$ is nonzero. 

By the Cauchy interlacing theorem, the minimal eigenvalue of $C$, $\lambda_{\min}(C)$, is less than or equal to $\lambda_{\min}(T)$. We will show the strict inequality
\[\lambda_{\min}(C)<\lambda_{\min}(T).\]
For $\epsilon\in\R$, take
\[u=(\epsilon,v).\]
The Rayleigh quotient $R(C,u)$ equals
\[R(C,u)=\frac{\langle Cu,u\rangle}{\|u\|^2}=\frac{\epsilon^2 \Real\,w+ \epsilon\sqrt{1-|w|^2}v_1+\lambda_{\min}(T)}{1+\epsilon^2}.\]
For sufficiently small nonzero $\epsilon$ with sign opposite to that of  $v_1$, we get the inequality
\[\lambda_{\min}(C)\le R(C,u)<\lambda_{\min}(T)=-\cos\frac{\pi}{d}.\]
On the other hand, by (\ref{E55}) and the assumption $\alpha\le \pi-\pi/d,$
\[\lambda_{\min}(C)\ge\cos \alpha\ge \cos\left(\pi-\frac{\pi}{d}\right)=-\cos\frac{\pi}{d},\]
 which is a contradiction.
\end{proof}

\end{document}